\documentclass[12pt,reqno]{amsart}
\usepackage[centertags]{amsmath}
\usepackage{amssymb}
\usepackage{amsthm}
\usepackage{amsfonts,amssymb}

\usepackage{mathrsfs}
\usepackage{graphicx}

\theoremstyle{plain} 
\newtheorem{cor}{Corollary}
\newtheorem{lem}{Lemma}

\newtheorem{thm}{Theorem}

\theoremstyle{definition}

\theoremstyle{remark}

 \DeclareMathOperator{\diag}{diag}

\DeclareMathOperator{\pd}{\partial}

\DeclareFontFamily{OT1}{pzc}{}
\DeclareFontShape{OT1}{pzc}{m}{it}{<-> s * [1.10] pzcmi7t}{}
\DeclareMathAlphabet{\mathpzc}{OT1}{pzc}{m}{it}

\DeclareMathSymbol{\R}{\mathalpha}{AMSb}{"52}
\DeclareMathSymbol{\C}{\mathalpha}{AMSb}{"43}
\newcommand{\mbb}[1]{\mathbb{#1}}

\newcommand{\K}{\mbb{K}}

\newcommand*{\Scale}[2][4]{\scalebox{#1}{$#2$}}%

\newcommand{\set}[1]{\left\{#1\right\}}
\newcommand{\comment}[1]{}
\newcommand{\ra}{\rightarrow}

\newcommand{\pzcm}{\mathpzc{m}}
\newcommand{\fg}{\mathfrak{g}}
\newcommand{\fq}{\mathfrak{q}}

\newcommand{\bv}{\mathbf{v}}

\newcommand{\by}{\mathbf{y}}
\newcommand{\bw}{\mathbf{w}}

\newcommand{\mcm}{\mathcal{M}}

\newcommand{\mcs}{\mathcal{S}}

\newcommand{\beqn}{\begin{equation}}
\newcommand{\eeqn}{\end{equation}}
\newcommand{\balign}{\begin{align}}
\newcommand{\ealign}{\end{align}}
\newcommand{\bsube}{\begin{subequations}}
\newcommand{\esube}{\end{subequations}}

\begin{document}



\title[]{Algebraic structure and maximal dimension of the symmetry algebra for arbitrary systems of ODEs}

\author[]{J.C. Ndogmo}


\address{Department of Mathematics and Applied Mathematics\\
University of Venda\\
P/B X5050, Thohoyandou\\
Limpopo  0950, South Africa}

\email{jean-claude.ndogmo@univen.ac.za}

\begin{abstract}
It is known  for scalar ordinary differential equations, and for systems of ordinary differential equations of order not higher than the third, that their Lie point symmetry algebras is of maximal dimension if and only if they can be reduced by a point transformation to the trivial equation $\mathbf{y}^{(n)}$=0.  For arbitrary systems of ordinary differential equations of order $n \geq 3$ reducible by point transformations to the trivial equation, we determine the complete structure of their Lie point symmetry algebras as well as that for their variational, and their divergence symmetry algebras. As a corollary,  we obtain the maximal dimension of the Lie point symmetry algebra for any system of linear or nonlinear ordinary differential equations.
\end{abstract}


\maketitle

\section{Introduction}
\label{s:intro}

It has long been established \cite{liecanonic} for scalar $n$th order ordinary differential equations ({\small \sc ode}s) that they are of maximal Lie point symmetry  if and only if they can be reduced by an invertible point transformation to the canonical form $y^{(n)}=0$ (see also
\cite[Theorem 14]{olver94} and \cite[Theorems 6.39 and 6.43]{olver-EIS}).  Krause and Michel \cite{KM} also showed a similar result for \emph{scalar} linear equations, and that for this type of equations being of maximal symmetry is in fact equivalent to being iterative. These equivalence relationships point to some intriguing properties that {\small \sc ode}s of maximal symmetry should entertain. Lie \cite{liecanonic} showed that for second-order scalar  {\small \sc ode}s, the Lie algebra of any given equation of maximal symmetry is isomorphic to  $\mathfrak{sl}_3,$ and Gonz\'alez-L\'opez \cite{lopez-slde} extended this result to systems of $\mathpzc{m}$ linear second-order {\small \sc ode}s by showing that any such system is of maximal symmetry if and only if it's symmetry algebra is isomorphic to $\mathfrak{sl}_{\mathpzc{m}+2}.$\par

The structure of the symmetry algebra of linear equations of maximal symmetry has also been found \cite{KM} for scalar equations of order $n \geq 3,$ and such a result is based on another result of Lie \cite{lie-dn} which determines the maximal dimension of the symmetry algebra for any scalar {\small \sc ode} of a general order. However, the maximal dimension of the symmetry algebra is not known for systems of {\small \sc ode}s of arbitrary dimensions, and not even in the restricted case of linear systems,  although some upper bounds have been suggested. Indeed, Gonz\'alez-Gasc\'on and Gonz\'alez-L\'opez \emph{estimated}  in \cite{gonzal-newres} \emph{upper bounds} for the dimension of the symmetry algebra of systems of {\small \sc ode}s  of the general form
\begin{equation} \label{gen-sde}
\mathbf{y}^{(n)} = F(x, \mathbf{y}_{(n-1)}), \qquad \mathbf{y} \in \R^{\mathpzc{m}},
\end{equation}
where $\mathbf{y}_{(n-1)}$ denotes $\mathbf{y}= \mathbf{y}(x)$ and all its derivatives up to the order $n-1$.  The values thus found for these upper bounds were $\mathpzc{m}^2 + 4\mathpzc{m} +3$ for $n=2$ and $\mathpzc{m}^2 + n\mathpzc{m} +3,$ for $n \geq 3.$ The upper bound value found for $n=2$ had also been obtained by Marcus  \cite{markus} in an earlier study. Moreover, according to a result of Fels  \cite{fels-bk, bk014}  a system of {\small \sc ode}s \eqref{gen-sde} of order $n\leq 3$ is of maximal symmetry if and only if it is reducible by an invertible point transformation to the form $\by^{(n)}=0,$ which we shall refer to as the \emph{canonical form}.  The corresponding class of equivalent equations under invertible point transformations will be called the \emph{canonical class}. Thus for $n \geq 4$ it has not yet been established when a system of {\small \sc ode}s has maximal symmetry nor what the maximal dimension of the symmetry algebra is.\par

In this paper, we determine for the order $n\geq 3,$ the structure of the symmetry algebra for systems of equations   in any given canonical class, i.e. for those \emph{systems of linear or nonlinear} {\small \sc ode}s that are equivalent by a point transformation to a canonical form.  As a corollary, and based on the cited result of \cite{gonzal-newres}, we obtain the maximal dimension of the symmetry algebra for any system of {\small \sc ode}s of the form  \eqref{gen-sde}.
We also determine the structure of the Lie algebras of variational symmetries and divergence symmetries for these systems. The connection between the structure of the Lie point symmetry algebra thus found for the order $n \geq 3$ and that found in \cite{lopez-slde} for the corresponding second-order systems if fully described.  In particular, we show that the two structures only differ by an additional subalgebra of so-called non-Cartan symmetries  that second-order equations possess. These non-Cartan symmetries are known only for scalar linear equations \cite{moyoL, charalambous}, and we obtain their generalization to arbitrary dimensions . In the sequel the field of scalars $\K$ of the system of equations to be studied will be assumed to be an arbitrary field of characteristic zero.

\section{The Lie point symmetry algebra}
\label{s:liesym}

Under a given diffeomorphism $\sigma$ of the space $\mcm= \K \times \K^{\pzcm}$ of the independent variable $x$ and dependent variables $\by= (y_1, \dots, y_\pzcm),$  one has  $\sigma_*[\bv, \bw]=[\sigma_* (\bv), \sigma_*(\bw)]$ for any  vector fields $\bv$ and $\bw$ on $\mcm,$ where $\sigma_*$ denotes the push-forward of $\sigma.$   Consequently the symmetry algebras of equivalent equations are isomorphic and satisfy the same commutation relations in some given bases. To describe the structure of the symmetry algebras in the canonical classes, it should therefore be enough to restrict our attention to the canonical forms themselves. However,  for this study we shall rather consider the whole equivalence class of arbitrary canonical forms for linear systems of {\small \sc ode}s, as this yields much clearer and much general results.\par

The \emph{normal form} of a general element in a given canonical class can be obtained for linear systems of {\small \sc ode}s   by applying the most general element of the Lie pseudo-group of invertible point transformations preserving normal forms, also called the equivalence group (for normal forms), to the
corresponding  canonical equation $\by^{(n)}=0.$ The expression of an element of the equivalence group as obtained in \cite{char-nj} is given by
\beqn  \label{eqvnor}
x= f(z),\qquad  \by = f'(z)^{(n-1)/2}\, C \bw,
\eeqn
where $f$ is an arbitrary function and $C$ an arbitrary $\pzcm \times \pzcm$ constant matrix. It turns out that this expression coincides with the well known and corresponding expression for scalar equations.  The renaming of variables in the transformed equations will often be assumed. Let
\beqn \label{noreq1}
y^{(n)} + A_n^2\, y^{(n-2)} + \dots + A_n^j\, y^{(n-j)} + \dots A_n^n\, y=0
\eeqn
be the corresponding transformation of the scalar equation $y^{(n)}=0$   under \eqref{eqvnor}, where the $A_n^j= A_n^j (x) \in \K$ are scalars. Then as shown in \cite{char-nj}, in its normal form an  element in a canonical class  of linear {\small \sc ode}s consists of $\pzcm$ copies of \eqref{noreq1}. That is, it is an \emph{isotropic system} of the form
\beqn \label{noreqM}
\by^{(n)} + A_n^2\, \by^{(n-2)} + \dots + A_n^j\, \by^{(n-j)} + \dots A_n^n\, \by=0,
\eeqn
and represents indeed the transformation of $\by^{(n)}=0$ under \eqref{eqvnor}, and with the same value for $f$ as for \eqref{noreq1}. Moreover, for any given isotropic system of the form \eqref{noreqM} and with given arbitrary coefficients $A_n^j,$  we know  how to find the corresponding invertible transformation \eqref{eqvnor} mapping this system to the canonical equation $\by^{(n)}=0.$  Indeed, let
\beqn \label{srce}
y'' + \fq y=0,
\eeqn
where $\fq= A_2^2,$ be the second-order scalar equation obtained as a transformation of the free fall equation $y''=0$ under \eqref{eqvnor} and with the same function $f$ as for \eqref{noreq1}. We thus refer to \eqref{srce} as the second-order \emph{source equation} for \eqref{noreq1}. Let $u=u(x)$ be a nonzero solution of \eqref{srce} and denote by $\diag (\lambda_1, \dots, \lambda_p)$ a diagonal matrix with diagonal entries $\lambda_1, \dots, \lambda_p.$ Then the more specific transformation \eqref{eqvnor} of the form
$$
z= \int \frac{dx}{u^2},\qquad  \bw = u^{1-n} \diag (\lambda_1, \dots, \lambda_\pzcm)\, \by,
$$
where the $\lambda_j$ are scalars, maps $\bw^{(n)}=0$ to the given equation \eqref{noreqM}, and this follows from  the fact that \eqref{noreqM} is isotropic. On the other hand, if $u$ and $v$ are two linearly independent solutions of \eqref{srce}, then $n$ linearly independent solutions of the scalar equation \eqref{noreq1} are given by
\beqn \label{sk}
s_k= u^{n-1-k} v^k, \qquad \text{for $k=0, \dots, n-1.$}
\eeqn
Moreover the Wronskian  $\mathpzc{w} (u,v)$ of $u$ and $v$ is a nonzero constant since \eqref{srce} is in normal form, and we shall assume this constant to be normalized to one.\par

  Before we state and prove the main result (Theorem \ref{struct}), we wish to make some few more remarks. Let
\beqn \label{hij-skj}
H_{ij}= y_i \pd_j, \qquad \text{ and }\qquad  S_{kj} = s_k \pd_j
\eeqn
be vector fields on $\mcm,$ where $\pd_j= \pd_{y_j}$ and $s_k$ is a solution of \eqref{noreq1} as given by \eqref{sk}. The one-parameter subgroup generated by $H_{ij}$ is given by
\begin{subequations} \label{grpHij}
\begin{align}
z& =x,  \text{ and }  \\
w_p &=  \delta_{pj}\, t y_i + y_p,\qquad \text{ for $p=1, \dots, \pzcm$}
\end{align}
\end{subequations}
where $t$ is the group parameter and $\delta$ represents the Kronecker delta. Owing to the fact that the system \eqref{noreqM} is both linear and isotropic, it is clear that \eqref{grpHij} is a symmetry transformation for this system. On the other hand the group transformation generated by $S_{kj}$ is given by
\begin{subequations} \label{grpSkj}
\begin{align}
z& =x,  \text{ and }  \\
w_p &=   \delta_{pj}\, t s_k + y_p,\qquad \text{ for $p=1, \dots, \pzcm$},
\end{align}
\end{subequations}
and these transformations also leave \eqref{noreqM} invariant owing to the fact that the system is not only linear, but also isotropic. In the sequel all Lie algebras considered will be assumed to have base field $\K.$ Thus in particular $\mathfrak{gl}_m$ will mean $\mathfrak{gl}(m, \K),$ etc. The symbols $\oplus$ and $\dotplus$ will also represent the direct sum of Lie algebras and the direct sum of vector spaces, respectively. Denote by $\mathfrak{g}_H$ and $\mathfrak{g}_S$ the vector spaces generated by the $H_{ij}$ (for $i,j= 1, \dots, \pzcm$) and the $S_{k,j}$ (for $k=0, \dots, n-1, \text{ and } j=1, \dots, \pzcm$), respectively.\par
\begin{lem} \label{hij-hpq}
\begin{enumerate}\mbox{}
\item[\rm{(a)}] $\mathfrak{g}_H$ is a  Lie algebra with commutation relations
\beqn \label{comhij}
[H_{ij}, H_{pq}] = \delta_{pj} H_{iq} - \delta_{iq} H_{pj}.
\eeqn
In particular, $\mathfrak{g}_H$ is a reductive Lie algebra isomorphic to $\mathfrak{gl}_\pzcm,$ and having as a Lie algebra the direct sum decomposition
$$
\mathfrak{g}_H =  \mathfrak{s}_H \oplus \K H_0,
$$
where $H_0= \sum_i H_{ii}$ generates the center of $\mathfrak{g}_H$ and $\mathfrak{s}_H$ is isomorphic to $\mathfrak{sl}_\pzcm,$ and has generators the $H_{ij}$ with $i\neq j$ together with the $\pzcm -1$ vectors $H_{ii} -H_{i+1,\, i+1},$ for $i=1, \dots, \pzcm-1.$
\item[\rm{(b)}] The space $\mathfrak{g}_S$  is an abelian Lie algebra, and $\mathfrak{g}_H \dotplus \mathfrak{g}_S$ is also a Lie algebra of dimension $\pzcm^2 + n \pzcm$.
\end{enumerate}
\end{lem}
\begin{proof}
One can  see why \eqref{comhij} holds by considering the cases $j=q$ and $j \neq q,$  and by applying the simple rules for the Lie bracket of vector fields. On the other hand, \eqref{comhij} are the same as the commutation relations of $\mathfrak{gl}_\pzcm$ in its standard basis $\set{e_{ij}},$ where $e_{ij}$ is the $\pzcm \times \pzcm$ (elementary) matrix with entry $1$ in the $i$th row and the $j$th column and $0$ elsewhere. Thus the map $\varphi (H_{ij})= e_{ij}$ defined from $\mathfrak{g}_H$ to $\mathfrak{gl}_\pzcm$ is a Lie algebra isomorphism, as the $\pzcm^2$ vector fields  $H_{ij}$ clearly also form a basis of $\mathfrak{g}_H.$ Since $\mathfrak{gl}_\pzcm$ is reductive, the rest of part (a) of the lemma follows from the well-known construction \cite{hump} of a basis for $\mathfrak{sl}_\pzcm$ and for the center of $\mathfrak{gl}_\pzcm,$ as well as the decomposition of $\mathfrak{gl}_\pzcm$ as direct sum of $\mathfrak{sl}_\pzcm$ and its one-dimensional center generated by $\sum_i e_{ii}.$\par
It is also straightforward to note that the $S_{kj}$  satisfy the  commutation relations $[S_{p\, i}, S_{qj}]=0$ for all $i,j$ and for all $p,q,$ and thus $\mathfrak{g}_S$ is abelian. On the other hand we have $[ S_{k,p}, H_{ij}] = \delta_{pi} S_{kj}$ and as the $\pzcm^2 + n \pzcm$ vectors $H_{ij}$ and $S_{kp}$ are clearly linearly independent,  this proves the rest of part (b) and completes the proof of the lemma.
\end{proof}

Denote by $\mathfrak{g}$ the (Lie point) symmetry algebra of \eqref{noreqM}, and let $\mathfrak{g}_F$   be the complement subspace of $(\mathfrak{g}_H \dotplus \mathfrak{g}_S)$ in $\mathfrak{g}.$ Note that a solution of the isotropic system \eqref{noreqM} is a vector $(s^1, \dots, s^\pzcm)$ where each $s^j$ is a solution of the scalar equation \eqref{noreq1}.
\begin{lem} \label{gf}
A vector field in $\mathfrak{g}_F$ has the form
\beqn \label{gengf}
\bv =  \xi \pd_x + \Scale[0.90]{ \sum_i} \phi \, y_i \pd_i,
\eeqn
for some functions $\xi=\xi(x)$  and $\phi = \phi(x).$ Moreover, $\xi \neq 0.$
\end{lem}
\begin{proof}
Let $\bv= \xi(x, \by) \pd_x + \Scale[0.90]{ \sum_i} T_i(x, \by) \pd_i$ be  in $\mathfrak{g}_F.$  Then for each generator $S_{kj}$ in $\mathfrak{g}_S,$  one has
$[S_{kj}, \bv]= \xi_j \pd_x + \Scale[0.90]{ \sum_i} T_{ij} \pd_i,$ where
\begin{align} \label{fjTij}
\begin{split}
\xi_j & = s_k \frac{\pd \xi}{\pd y_j}, \qquad \text{for $j=1, \dots, \pzcm.$} \\
T_{ij}&= \begin{cases} s_k \frac{\pd T_i}{ \pd y_j} & \text{, for $i \neq j$} \\
                       s_{k-1} \left[ \xi[(1+ k -n) v u' - k u v'] + u v \frac{\pd T_j}{ \pd y_j}  \right] & \text{, otherwise.}    \end{cases}
\end{split}
\end{align}
Since the symmetry group of the equation transforms a solution into a solution, $\mathfrak{g}_S$ must be an ideal in $\mathfrak{g}.$ It thus follows from \eqref{fjTij} that $\xi=\xi(x)$ while $T_i$ must be linear in the $y_k,$ for $k=1, \dots, \pzcm.$ That is,
$
T_i= a_i+ \Scale[0.90]{ \sum_k} a_{ik}\,  y_k
$
for some functions $a_i= a_i(x)$ and $a_{ik}= a_{ik} (x).$ Consequently,
\begin{align}
\comment{[H_{ij}, \bv] & = (y_i a_{jj} - T_i) \pd_j + \Scale[0.90]{  \sum_{k \neq j}} y_i a_{kj} \pd_k,\quad \text{ for all $i,j$}  \label{[hijbv]1} \\
}
[H_{pq}, [H_{ij}, \bv]] &= y_p (\delta_{qi} a_{jj} - a_{iq} ) \pd_j - y_i a_{pj} \pd_q + \Scale[0.90]{  \sum_{k\neq j}} \delta_{qi} y_p a_{kj} \pd_k, \label{[hijbv]}\\
   & \equiv \bv_{pqij}  \notag
\end{align}
for all $i,j$ and $p,q$ with $p\neq j.$  In particular,
$$\bv_{piij} =  y_p (a_{jj} - a_{ii}) \pd_j - y_i a_{pj} \pd_i +\Scale[0.90]{  \sum_{k \neq j}} y_p a_{kj} \pd_k .$$
For $\bv_{piij}$ with $p\neq i$ to be a symmetry (vector), one sees that the scalars $a_{jj}- a_{ii}$ and $a_{kj}$ must be constants for all $i,j$ and for all $k\neq j.$ Substituting these values in the corresponding expressions for
$$\bv= \xi\pd_x+ \Scale[0.90]{  \sum_i} \big(\Scale[0.90]{   \sum_k} a_{ik}\, y_k + a_i \big)\pd_i$$
shows that   $\bv$ is a symmetry  only if these constants are all zeros, so that $\bv= \xi\pd_x + \Scale[0.84]{  \sum_i} (a_{11}y_i + a_i )\pd_i. $ Then $[\bv, H_{jj}]= a_j \pd_j, $ and this is a symmetry  if and only if $a_j$ is a solution to \eqref{noreq1}, and we should thus assume that $a_j=0$ for all $j=1, \dots, \pzcm.$ \comment{ as $\mathfrak{g}_F$ is the complement subspace of $\mathfrak{g}_S \dotplus \mathfrak{g}_H.$} This proves the first part of the lemma with $\phi= a_{11}.$ For $\xi=0,$ one has $[\bv, S_{kj}]= \phi\, S_{kj},$ so that $\phi\, s_k$ must be a solution of \eqref{noreq1} for all $k,$ and this forces each $\phi$ to be a constant and $\bv$ to lie in $\mathfrak{g}_H,$ which is an excluded possibility. So $\xi \neq 0$ and this completes the proof of the lemma.
\end{proof}

Denote by $\mathfrak{X}(\K)$ the Lie algebra of vector fields on $\K.$  Note also that the third order scalar equation with source equation \eqref{srce}, i.e. the equation obtained by transforming $y^{(3)}=0$ under \eqref{eqvnor} with the same function $f$ as that for \eqref{srce} is given by
\beqn \label{iter3}
y^{(3)} + 4 \fq y' + 2 \fq' y=0.
\eeqn
\begin{lem} \label{xi}
Any function $\xi=\xi(x)$  for which the vector $\bv= \xi \pd_x + \Scale[0.84]{  \sum_i} \phi y_i \pd_i$ is a generator of $\mathfrak{g}_F$  is  such that $\xi \pd_x$ is a generator of the subalgebra of $\mathfrak{X} (\K)$ isomorphic to  $\mathfrak{sl}_2.$ In particular, a maximum of three linearly independent such functions is given by $v^2, 2 u v,$ and $-u^2.$
\end{lem}

\begin{proof}
 It clearly follows from the form of elements in $\mathfrak{g}_H$ and  in $\mathfrak{g}_S,$ as well as from  the form of elements in $\mathfrak{g}_F$ as specified by Lemma \ref{gf} that the kernel of the Lie algebra homomorphism  $\rho \colon \mathfrak{g} \ra \mathfrak{X} (\K)$ given by
$\rho [\xi \pd_x + \Scale[0.84]{ \sum_i} (\phi y_i + s^i)  \pd_i ] = \xi \pd_x$ is precisely the subalgebra $\mathfrak{g}_H \dotplus \mathfrak{g}_S$ of $\mathfrak{g}.$ It is well-known from a result of Lie \cite{lieTransf,KM}, that all finite dimensional subalgebras of $\mathfrak{X} (\K)$ are isomorphic to $\mathfrak{sl}_2$ or to one of its subalgebras. If we let $\set{\xi_i (x) \pd_x\colon 1 \leq i \leq d} $ be a basis of  such a subalgebra of $\mathfrak{X}(\K),$ then $d \leq 3.$  Indeed, for $d>1$ the functions $\xi_i$ form a basis of the solution space of the $d$th order {\small \sc ode}
\beqn \label{wronskeq}
\frac{1}{\mathpzc{w}(\xi_1, \dots, \xi_d)} \, \mathpzc{w}(y, \xi_1, \dots, \xi_d)=0.
\eeqn
Since all Wronskians $\mathpzc{w} (\xi_i, \xi_j)$ should also be solutions to this equation, it follows that $d \leq 3$  as claimed. For arbitrary linearly independent  functions $\alpha$ and $\beta,$ the Lie algebra generated by $\alpha^2 \pd_x,\, -\beta^2 \pd x$ and $2 \alpha \beta \pd_x$ is clearly isomorphic to $\mathfrak{sl}_2.$ In particular if we let $u$ and $v$ be the usual solutions of the source equation \eqref{srce}, then for $\xi_p= v^2,$  $\xi_m= -u^2,$ and $\xi_z= 2 u v$ the corresponding equations \eqref{wronskeq} reduces exactly to \eqref{iter3}. That is, the linearly independent functions $\xi$ are a subset of the basis of the solution space of the third order scalar equation \eqref{iter3}.
\end{proof}

This proof is similar to that  given in \cite{KM} for establishing a similar result in the case of scalar equations, but it is stated and proved here quite differently. Although we know how to find the linearly independent values of $\xi$ in \eqref{gengf}, that alone does not tell how to determine the function $\phi,$ and one easy way to completely determine the form of a generator $\bv$ in $\mathfrak{g}_F$ is to make use of the infinitesimal criterion of invariance for a symmetry generator \cite{olver1}.  If we denote by $\bv^{(n)}$ the $n$th prolongation of a vector field $\bv$ and by $\Delta \equiv (\Delta_1, \dots, \Delta_\pzcm)=0$ a system of  differential equations, the infinitesimal invariance condition is given by $$\bv^{(n)} (\Delta_\nu)\big \vert_{(\Delta=0)} =0, \text{ for all $\nu=1, \dots, \pzcm.$}$$
Writing this condition and expanding the resulting expression as a polynomial in the derivatives of the dependent variables yields the determining equations for $\bv.$ In the actual case of the vector $\bv$ in \eqref{gengf} which is almost already completely determined, the determining equations are fairly simple and the relevant ones reduce to
\begin{subequations} \label{detv}
\begin{align}
\phi'- \frac{n-1}{2}\xi'' &=0   \label{detphi}\\
\xi^{(3)} + 4 \fq \xi' + 2 \fq' \xi&=0 \label{detxi}
\end{align}
\end{subequations}
We note that condition \eqref{detxi} above is just another restatement of part of Lemma \ref{xi}. The resulting expression for $\bv$ in \eqref{gengf} takes the form
\beqn \label{exprV}
\bv = \xi \pd_x +  \Scale[0.90]{ \sum_{i=1}^\pzcm} \left(\frac{n-1}{2} \xi' + k_1 \right) y_i \pd_i,
\eeqn
where $\xi$ is a solution  of the third order equation \eqref{iter3} and where $k_1$ is an arbitrary constant of integration which by the nature of $\mathfrak{g}_F$ should be omitted. It turns out that $\mathfrak{g}_F$ has the maximal dimension $3.$ Indeed, let
\begin{align}  \label{Fpmz}
\begin{split}
F_p &=  v^2 \pd_x + (n-1) v v'\, \Scale[0.90]{ \sum_{i=1}^\pzcm} y_i \pd_i \\
F_m &=  -u^2 \pd_x - (n-1) uu'\, \Scale[0.90]{ \sum_{i=1}^\pzcm} y_i \pd_i \\
F_z &=  2 u v \pd_x + (n-1) (u v' + u' v )\, \Scale[0.90]{ \sum_{i=1}^\pzcm} y_i \pd_i,
\end{split}
\end{align}
then we have the following result.
\begin{thm} \label{struct}
\begin{enumerate}\mbox{}
\item[\rm{(a)}] The complement subspace $\mathfrak{g}_F$ of the subalgebra $\mathfrak{g}_H \dotplus \mathfrak{g}_S$ of the Lie point symmetry algebra $\mathfrak{g}$ of \eqref{noreqM} is a three-dimensional subalgebra isomorphic to $\mathfrak{sl}_2$ and having generators $F_p, F_m,$ and $F_z.$

\item[\rm{(b)}] The Lie algebra $\mathfrak{g}$ has dimension $\pzcm^2 + n \pzcm + 3,\quad (n\geq 3),$ and Levi decomposition
              $$ \mathfrak{g} = (\mathfrak{g}_F \oplus \mathfrak{s}_H)\ltimes (\mathfrak{g}_S \oplus \K H_0),  $$
where the symbol $\ltimes$ represents the semidirect sum of Lie algebras, and $(\mathfrak{g}_F\oplus \mathfrak{s}_H)$ the Levi factor.
\end{enumerate}
\end{thm}

\begin{proof}
By construction, linearly independent generators of $\mathfrak{g}_F$ should be of the form \eqref{Fpmz}, and it is easily verified using the infinitesimal criterion of invariance that the three vector fields $F_p, F_m,$ and $F_z$ are indeed symmetries of the  system of equations \eqref{noreqM}. Moreover, they satisfy the $\mathfrak{sl}_2$ commutation relations $[F_p, F_m]= F_z,\; [F_z, F_p]= 2 F_p$ and $[F_z, F_m]= -2 F_m,$ and this proves part (a). For part (b), this follows from the fact that $[\mathfrak{g}_H, \mathfrak{g}_F]=0$ and the already known facts that $\mathfrak{g}_S$ is an abelian ideal in $\mathfrak{g},$ while $\K H_0$ is the center of $\mathfrak{g}_H$ and $\mathfrak{s}_H$ is also semisimple. This completes the proof of the theorem.
\end{proof}

It is also interesting to note that the action of $\mathfrak{g}_F$ on $\mathfrak{g}_S$ corresponds to a semisimple representation which is irreducible on each of the $n$-dimensional subalgebras $\mathfrak{g}_{Sj}$ of $\mathfrak{g}_S$ generated by the $S_{kj},$ for $j$ fixed. This follows from the fact that the commutation relations between $\mathfrak{g}_F$ and $\mathfrak{g}_S$ are given by
\begin{align*}
[F_p, S_{kj}] &= (1+k-n) S_{k+1,\, j}, \quad [F_m, S_{kj}]= k S_{k-1,\, j}, \\
[F_z, S_{kj}] &= (1+2k-n) S_{kj},
\end{align*}
for all $j=1, \dots, \pzcm$ and for all $k=0, \dots, n-1.$ For convenience the complete table of commutation relations of $\mathfrak{g}$ is given in Table \ref{t:1}. \par

\begin{table}[h]
 \caption{ \label{t:1} \protect \footnotesize Commutation relations for the symmetry algebra $\mathfrak{g}$ of \eqref{noreqM} }
\begin{tabular}{r | c c c c c }

              & $H_{ij}$                                 & $S_{ha}$                 & $F_p$         & $F_m$                & $F_z$       \\[1.0pt] \hline \\[2.2pt]

  $H_{pq}$  & $\delta_{iq} H_{pj} - \delta_{pj} H_{iq} $ &  $-\delta_{ap} \,S_{hq}$ & $0$           & $0$                  & 0      \\[2.2pt]

  $S_{kb}$  & $\delta_{bi} S_{kj}$                        &  $0$                   & $-$             & $-$                 & $-$      \\[2.2pt]

   $F_p$    & $0$                                         &  $(1+h-n)S_{h+1,a}$    & $0$              & $F_z$           & $- 2 F_p$       \\[2.2pt]
  $F_m$     & $0$                                         &  $ hS_{h-1, a}$        & $-F_z$        & $0$           & $2 F_m$       \\[2.2pt]

  $F_z$     & $0$                                        &  $(1+2h-n)S_{h,a}$      & $2 F_p$        & $-2F_m$           & $0$       \\ \hline

\end{tabular}

\parbox[t]{0.9\textwidth}{ \vspace{3mm} \protect \scriptsize   The  basis vectors in the top row and leftmost column are not the same, as $H_{ij}$ and $S_{ha}$ in the row vectors are replaced with $H_{pq}$ and $S_{kb}$ in the column vectors, respectively. Thus the table is not skew symmetric. The commutation relations $[S_{kb}, F_i],$ where $F_i$ is a basis element of $\mathfrak{g}_F$ is represented in the table with a dash as they can be obtained from those of the form $[F_i, S_{ha}]$ given in the table.
}
\end{table}

Although Theorem \ref{struct} has been stated only for systems of order $n \geq 3,$ according to a result of \cite{lopez-slde}, the maximal dimension of the symmetry algebra for second-order systems of linear equations is $d_{\pzcm , 2} = \pzcm^2 + 4 \pzcm + 3.$ For $n=2$ this can be written as  $~(\pzcm^2 + n \pzcm + 3) + n \pzcm .$ In other words, the expression for the maximal dimension $d_{\pzcm , 2}$ is that for systems of order $n \geq 3$ plus the additional term $n \pzcm.$ Moreover it is easy to verify that if we let $n=2$ in Theorem \ref{struct}, the Lie algebra $\mathfrak{g}$ is still defined and its corresponding $\pzcm^2 + n \pzcm + 3$ generators are symmetries of the associated second order equation. The additional $2 \pzcm$ symmetries $C_{ik}$ of the second-order equation are a generalization the so-called non-Cartan symmetries known for second-order scalar linear equations \cite{moyoL, charalambous}. If we denote by $u_1= u$ and $u_2= v$ the two linearly independent solutions of the source equation \eqref{srce}, then one can verify that these non-Cartan symmetries are given for a  second-order equation in normal form \eqref{noreqM} by
\begin{equation} \label{ncartansym}
C_{ik} = y_i u_k \pd_x + \sum_{j=1}^\pzcm y_i y_j u_k' \pd_j,\quad \text{for  $i= 1, \dots, \pzcm$ and $k=1,2,$}
\end{equation}
where  $u_k' = d u_k/ dx.$  These $2 \pzcm$ non-Cartan symmetries form and abelian algebra, but together with the other $\pzcm^2 + 2 \pzcm + 3$ symmetry generators of  the corresponding Lie algebra  $\mathfrak{g}$ of \eqref{noreqM} (with $n=2$), they form as already noted the semisimple Lie algebra $\mathfrak{sl}_{\mathpzc{m}+2}.$\par
We also note that on the basis of the upper bound obtained in  \cite{gonzal-newres} for systems of {\small \sc ode}s of the general form \eqref{gen-sde}, the results obtained for all orders $n \geq 3$ in  Theorem \ref{struct} of this paper  have the following important consequence.
\begin{cor}\label{maxdimodes}
For $n \geq 3,$ the maximal dimension of the Lie point symmetry algebra for systems of {\small \sc ode}s of the general for \eqref{gen-sde} is $\pzcm^2 + n \pzcm + 3,$ and this maximal dimension is achieved for all equations in canonical classes of linear systems of {\small \sc ode}s.
\end{cor}
In fact, it was also wrongly pointed out in \cite{gonzal-newres} that the  upper bound $\pzcm^2 + n \pzcm + 3$ proposed in that paper is achieved only for $\pzcm=1,$ i.e. only for scalar equations.  On the other hand as already noted, the maximal dimension of the symmetry algebra for \eqref{gen-sde} is well-known for $n \leq 3$  \cite{fels-bk}.

It should be noted that the class of canonical forms having as symmetry algebra up to isomorphism the Lie algebra $\mathfrak{g}$  determined in this section is quite general. The expression of a general element from this class can be obtained for a given order of the equation by applying the most general point transformation to the associated canonical form. For second-order equations in dimension 2, if we let
\begin{equation*}
x= \theta(t, w_1, w_2),\quad Y= \psi(t, w_1, w_2),\quad   Z= \omega(t, w_1, w_2)
\end{equation*}
be such an invertible point transformation, then the expression of an arbitrary element in the corresponding class of equations takes the form
\begin{align*}
\begin{split}
 &-\left(\psi _x+Y_x \psi _Y+Z_x \psi _Z\right) \big(\theta _Y Y_{x,x}+\theta _Z Z_{x,x}+\theta _{x,x}+2 Y_x \theta _{x,Y}+Y_x^2 \theta _{Y,Y}\\
& +2 Z_x \left(\theta _{x,Z}+Y_x \theta _{Y,Z}\right)+Z_x^2 \theta _{Z,Z}\big)+\left(\theta _x+Y_x \theta _Y+Z_x \theta _Z\right) \big(\psi _Y Y_{x,x} +\psi_Z Z_{x,x}\\
 &+\psi _{x,x}+2 Y_x \psi _{x,Y}+Y_x^2 \psi _{Y,Y}+2 Z_x \left(\psi _{x,Z}+Y_x \psi _{Y,Z}\right)+Z_x^2 \psi _{Z,Z}\big)=0
\end{split}\\[4mm]
 \begin{split}
&-\left(\omega _x+Y_x \omega _Y+Z_x \omega _Z\right) \big(\theta _Y Y_{x,x}+\theta _Z Z_{x,x}+\theta _{x,x}+2 Y_x \theta _{x,Y}+Y_x^2 \theta _{Y,Y} \\
&+2 Z_x \left(\theta _{x,Z}+Y_x \theta _{Y,Z}\right)+Z_x^2 \theta _{Z,Z}\big)+\left(\theta _x+Y_x \theta _Y+Z_x \theta _Z\right) \big(\omega _Y Y_{x,x} +\omega
_Z Z_{x,x}\\
&+\omega _{x,x}+2 Y_x \omega _{x,Y}+Y_x^2 \omega _{Y,Y}+2 Z_x \left(\omega _{x,Z}+Y_x \omega _{Y,Z}\right)+Z_x^2 \omega _{Z,Z}\big)=0.
\end{split}
\end{align*}
Although this system, which is the simplest for an \emph{arbitrary element} in the class of vector {\small \sc ode}s, has quite a complicated expression we note that it is a differential polynomial of degree $3$ in the dependent variables $Y$ and $Z.$ This is a useful criteria for identifying them. A more complete identification of families of {\small \sc ode}s and the determination of their canonical forms can be directly achieved through the theory of invariants of differential equations \cite{invode1,invode2}, and the Cartan moving frame methods \cite{Pohjanpelto,OlverMF1}. Nevertheless, because of the intricate calculations usually involved in these methods,  Lie  symmetry methods remain as shown in this section some of the most efficient methods for identifying complex systems of equations as well as  for finding their canonical forms.
\section{The variational symmetry algebra}
\label{s:transfo}
Among the symmetries of a given system $\Delta \equiv(\Delta_1, \dots, \Delta_\pzcm)=0$ of differential equations, some of them so-called variational symmetries, and their slightly more general versions called divergence symmetries \cite{olver1} play a particularly important role as they determine amongst others the conservation laws, and first integrals thereof associated with the system. Variational symmetries may however exist only for certain even-order systems of equations. In order to clarify these two concepts we introduce some few new terms, by restricting however as much as possible our discussion to the system of  {\small \sc ode}s of the form \eqref{noreqM} that we are considering.\par

Denote by $f[\by]= f(x, \by^{(m)})$ a differential function depending on $x,\, \by= \by(x),$ and the derivatives of $\by$ up to an arbitrary but fixed order $m.$ Given an open and connected subset $\Omega \subset \K,$ with smooth boundary $\partial \Omega,$ the problem of finding the extremals of the functional
\begin{equation}\label{varprob}
\vartheta[\by] = \int_{\Omega} L(x, \by^{(m)}) dx
\end{equation}
is called a variational problem and the integrand $L= L(x, \by^{(m)})$ is referred to as the Lagrangian. Every smooth extremal $\by= \by(x)$ of \eqref{varprob} must be a solution of the associated Euler-Lagrange equations
\begin{subequations}\label{EL}
\begin{align}
E_j (L)&=0, \qquad j= 1, \dots, \pzcm \label{EL1}\\
\intertext{where in the actual case of systems of {\small \sc ode}s}
E_j &= \sum_k  (-1)^k D_x^k  \frac{\partial}{\partial y_j^{(k)}} \label{EL2}
\end{align}
\end{subequations}
is the $j$th Euler operator, $D_x$ is the total differential operator with respect to the independent variable $x$, while $y_j^{(k)}= D_x^k \,y_j =  d^k y_j / d x^k . $ The system of Euler-Lagrange equations \eqref{EL1} in which  the differential operators $E_j$  are given by \eqref{EL2} is of even order $n=2m,$ provided that $L$ is not linear in the variables $y_j^{(m)}.$ A differential equation $\Delta=0$ is called Lagrangian with respect to a Lagrangian $L(x, \by^{(m)})$ if it is {\em equivalent} to the Euler-Lagrange equation \eqref{EL1}, where  equivalent here means that the solutions set of the two equations are the same.\par

Suppose that $\Delta=0$ is actually an $n$th order system of {\small \sc ode}s and denote by
\begin{equation}\label{vDelta}
\mathbf{v} = \xi(x,\by) \frac{\partial }{\partial x} + \sum_{j=1}^\pzcm \psi_j (x,\by) \frac{\partial }{\partial y_j}
\end{equation}
an infinitesimal generator of the symmetry group of $\Delta=0,$ acting on an open set $M \subset \Omega \times \K^\pzcm,$  and set
\beqn \label{s(v)}
\mathscr{S} (\bv) = \bv^{(n)} (L) + L\, D_x\, \xi.
\eeqn
 For the equation $\Delta=0,$ or equivalently for the variational problem $\vartheta[\by] = \int_{\Omega} L(x, \by^{(m)}) dx,$  $\bv$ is a variational symmetry  if and only if
\begin{subequations} \label{varcdt}
\begin{align}
\mathscr{S} (\bv) &=0 \label{varsym}\\
\intertext{and a divergence symmetry if and only if}
\mathscr{D} (\bv) &\equiv E[\mathscr{S} (\bv)] =0 \label{divsym}
\end{align}
\end{subequations}
for all $(x, \by^{(n)})$ in the extended jet space $M^{(n)}$ of $M,$  where $E=(E_1, \dots, E_\pzcm)$ is the Euler operator. Condition \eqref{varsym} is just the well known infinitesimal variational symmetry criterion \cite{olver1} while the corresponding infinitesimal condition for divergence symmetries is $\mathscr{S} (\bv) = D_x B$ for a certain functional $B= B[\by].$ Thus condition \ref{divsym} follows from the fact that $B$ is a total divergence if and only if it is a null Lagrangian, which means that the associated Euler-Lagrange equations vanish identically.\par

The space of all variational symmetries of a given Lagrangian system forms a Lie algebra, and the same holds true for the space of divergence symmetries of any given system of equations. Once a basis $\set{\bv_i}$  of the Lie point symmetry algebra has been found for a given  system of equations, to find a basis of the Lie algebra of variational symmetries or the Lie algebra of divergence symmetries, one can test the corresponding infinitesimal condition \eqref{varcdt} on a general linear combination  $\bv$ of the $\bv_i.$ In performing such a test, one may exclude those basis vectors which already satisfy individually the test, owing to the linearity of the operators $\mathscr{S}$ and $\mathscr{D}.$ \par

Having found in the previous section the structure of the Lie algebra $\mathfrak{g}$ of symmetries for the class \eqref{noreqM} of normal forms for systems of linear equations, we are interested in finding here the structures of the Lie algebra $\mcs_{var}$ of variational symmetries and $\mcs_{div}$ of divergence symmetries for the same class of equations \eqref{noreqM} of \emph{even} orders. First of all one can check by direct calculation that any such Linear system is Lagrangian, as it satisfies the required necessary and sufficient condition of having a self-adjoint Fr\'echet derivative. The Lagrangian function we will choose is the standard one for linear systems $\Delta (x, \by^{(n)})=0$  (of even orders) which is given by $L_0= \frac{1}{2} \by \cdot \Delta.$ As is well known, such a Lagrangian can always be replaced by an equivalent one of order $n/2,$  and equivalent Lagrangian functions yield the same Euler-Lagrange equations.  \par

We also note that if two differential equations $\Delta_A=0$ and $\Delta_B=0$ are equivalent under an invertible point transformation $\sigma$ and if $\Delta_A=0$ is Lagrangian with respect to a Lagrangian $L_A,$ then $\Delta_B=0$ is also Lagrangian with respect to a certain Lagrangian function $L_B$ which has a simple expression in terms of $\sigma$ \cite{lopez-slde,sarlet22}. Moreover, if $\bv_b$ is a variational symmetry for $\Delta_B=0$ relatively to the Lagrangian $L_B,$ then $\bv_a= \sigma_* \bv_b$ is  also a variational symmetry of $\Delta_A=0$ with respect to $L_A.$ In other words, if  two equivalent differential equations  are Lagrangian with respect to two given Lagrangian functions, then their variational symmetry algebras with respect to these two Lagrangian functions are isomorphic.  A similar result applies to the divergence symmetries of two equivalent equations $\Delta_A=0$ and $\Delta_B=0$. Therefore the structure we shall find explicitly for the Lie algebras $\mcs_{var}$ and $\mcs_{div}$ corresponding to the class of linear systems \eqref{noreqM} and the Lagrangian function $L_0$ is actually valid for the more general class of systems of all {\small \sc ode}s  equivalent by a local diffeomorphism $\sigma$ to a canonical form, and with respect to a certain Lagrangian whose expression is known in terms of $L_0$ and $\sigma.$ \par

For the class of equations \eqref{noreqM} under study, we shall first determine the structure of $\mcs_{div}.$ It is easily found by a direct calculation that $\mathscr{D} (F_p) = \mathscr{D} (F_m) = \mathscr{D} (F_z)=0$ and that similarly one has $\mathscr{D} (S_{kj})=0$ for all $k=1,\dots, n$ and for all $j=1, \dots, \pzcm.$ Hence $\mathfrak{g}_F$ and $\mathfrak{g}_S$ are subalgebras of $\mcs_{div}.$ \par

Note that if we denote the scalar equation \eqref{noreq1} simply by $\mathscr{E}_n[y]=0,$ then the isotropic system \eqref{noreqM} takes the form $\mathscr{E}_n[y_j]=0,$ for $j=1, \dots, \pzcm,$ where $\by= (y_1,\dots, y_\pzcm).$  Now let $\bv= \sum_{ij} a_{ij} H_{ij}$ for some scalars $a_{ij} \in \K.$ If in this expression we set $a_{ji}= - a_{ij}$ for $j>i,$  the resulting expression for $ \mathscr{D} (\bv)$ reduces to
$$
a_{ii}\, \mathscr{E}_n [y_i]=0, \quad \text{for $i= 1, \dots, \pzcm,$}
$$
and this shows that $a_{ii}=0$ for all $i$ and that all rotations $R_{ij}= H_{ij} -H_{ji}$ with $i\neq j$ of the space $\K^\pzcm$ (or equivalently the elementary skew-symmetric endomorphisms of $\K^\pzcm$) are divergence symmetries. It is well known that these rotations generate the special orthogonal Lie algebra $\mathfrak{so}(\pzcm),$ which is a semisimple Lie algebra of dimension $\pzcm(\pzcm-1)/2.$ These results can be summarized as follows.
\begin{thm} \label{th:sdiv}
The Lie algebra $\mcs_{div}$ of divergence symmetries of \eqref{noreqM} is given by the Levi decomposition
\beqn \label{sdiv}
\mcs_{div} = \left[\mathfrak{g}_F \oplus \mathfrak{so}(\pzcm)\right] \ltimes \mathfrak{g}_S,
\eeqn
and has dimension $\frac{1}{2} (\pzcm^2 + \pzcm ( 2 n-1) + 6).$
\end{thm}

Since  $\mcs_{var}$  is a subalgebra of $\mcs_{div},$  to find the generators of $\mcs_{var}$ we now let the generic generator $\bv$ for the test of the infinitesimal invariance condition \eqref{varsym} be just a linear combination of the generators of $\mcs_{div}.$ It is  easily verified by a direct calculation that
$$
\mathscr{S} (F_p) = \mathscr{S}(F_m)= \mathscr{S}(F_z)=0,\quad  \text { and}\quad \mathscr{S}(R_{ij})=0\quad  \text{for all $i\neq j.$ }
$$
This shows that $\mathfrak{g}_F$ and $\mathfrak{so}(\pzcm)$ are subalgebras of $\mcs_{var}.$  Thus we are left with testing only the vectors from $\mathfrak{g}_S,$ and so we set $\bv= \sum_{kj} b_{kj} S_{kj}.$  In this case $\mathscr{S}(\bv)$ is a linear differential polynomial in $y_1, \dots, y_\pzcm,$ in which the coefficient of $y_j^{(n)}$ is precisely $\frac{1}{2} \sum_{k=0}^\pzcm b_{kj} s_k.$ This shows that all the $b_{kj} $ must vanish so that no vector in $\mathfrak{g}_S$ may be a variational symmetry.

\begin{thm} \label{th:svar}
The variational symmetry algebra $\mcs_{var}$ of the class of equations \eqref{noreqM} is the semisimple lie algebra given by
$$
\mcs_{var} = \mathfrak{g}_F \oplus \mathfrak{so}(\pzcm),
$$
and it thus has dimension $\frac{1}{2} (\pzcm^2 -\pzcm +6).$
\end{thm}

\section{Concluding remarks}

We wish to recall that although divergence symmetries satisfy a more relaxed condition than that required for a vector field to be a variational symmetry, each of them also determines a conservation law associated with the equation \cite{olver1}. Moreover,  in the case of {\small \sc ode}s conservation laws yield  true constants of motion represented by first integrals. These first integrals often play a crucial role  in the study of solutions and properties of differential equations, including questions related to stability \cite{nitsch05,rzeszotko05} or global existence \cite{chavarrigaJ05,gine012,hounieJ011}. On the other hand, the variational symmetry group, which is the Lie group generated by $\mcs_{var},$ leaves the integral \eqref{varprob} invariant, and each variational symmetry can also be used to reduce the order of the associated Lagrangian equations by two \cite{olver1}.\par

 The structure of the Lie algebra $\mcs_{div}$ of divergence symmetries of  systems of {\small \sc ode}s of order $n\geq 3$  in the canonical classes  which we found corresponds  exactly to a generalization to arbitrary orders of the structure obtained in \cite[Eq. (3.38)]{lopez-slde} for second-order equations. However, this Lie algebra is called in that paper the Lie algebra of variational symmetries. But it can first be easily verified that no symmetry vector containing as a term a  non-Cartan symmetry as obtained in \eqref{ncartansym} may be a divergence symmetry, and in particular a variational symmetry. This means that to get the dimension of $\mcs_{div}$ for the case of second-order equations, it suffices to replace $n$ by $2$ in Theorem \ref{th:sdiv}. When this is done we recover the dimension found in \cite[Theorem 3]{lopez-slde} as the dimension of what is called in that paper the variational symmetry algebra, but which is in fact by the very formula used for its determination in the same paper \cite[Eq. (3.8)]{lopez-slde}, just the divergence symmetry algebra.\par

On the basis of the upper bound given in \cite{gonzal-newres} for systems of {\small \sc ode}s,  the dimension of the Lie point symmetry algebra that we found for equations in the canonical class is indeed maximal for systems of {\small \sc ode}s of the form \eqref{gen-sde}.  The other problem that remains to be solved is whether any system of {\small \sc ode}s of a given order and dimension that is of maximal symmetry belongs necessarily, as it is the case for scalar {\small \sc ode}s and for systems of {\small \sc ode}s of order $n \leq 3$, to the canonical class of the corresponding trivial equation. In other words, it is still not known whether the Lie point symmetry algebra of any  system of {\small \sc ode}s of maximal symmetry is  necessarily isomorphic to the  Lie algebra $\fg$ of \eqref{noreqM} found in Theorem \ref{struct}.

\comment{
\section*{Acknowledgement}
Funding: This work was supported by the National Research Foundation of South Africa [grant number: 97822].
}


\end{document}